\documentclass[12pt]{amsart}
\usepackage{amsmath, amsaddr, calc}

\newtheorem{theorem}{Theorem}
\newtheorem{conjecture}{Conjecture}
\newtheorem{lemma}{Lemma}

\theoremstyle{definition}
\newtheorem*{definition}{Definition}
\newtheorem*{example}{Example}
\newtheorem*{kneadingexample}{Kneading Example}

\DeclareMathOperator{\SL}{\mathrm{SL}}

\newcommand{\OR}[1]{{\overrightarrow{#1}}}

\begin{document}

\title{Reducing Quadratic Forms by Kneading Sequences}

\author{Barry R. Smith}
\email{barsmith@lvc.edu}
\address{Department of Mathematical Sciences \\ Lebanon Valley College\\
Annville, PA, USA}

\begin{abstract}
We introduce an invertible operation on finite sequences of positive integers and call it ``kneading''.  Kneading preserves three invariants of sequences --  the parity of the length, the sum of the entries, and one we call the ``alternant''.  We provide a bijection between the set of sequences with alternant $a$ and parity $s$ and the set of Zagier-reduced indefinite binary quadratic forms with discriminant $a^2 + (-1)^s \cdot 4$, and show that kneading corresponds to Zagier reduction of the corresponding forms.  It follows that the sum of a sequence is a class invariant of the corresponding form.  We conclude with some observations and conjectures concerning this new invariant.
\end{abstract}

\maketitle

\section{Kneading sequences}\label{S:intro}

We \emph{pinch} the left end of a finite sequence of positive integers by transforming it through the rule:
\begin{equation*}
	(x, y, z, \ldots) \mapsto
	\begin{cases}
		 (1, x-1, y, z, \ldots), \quad &\text{if $x \geq 2$,}\\
		 (y+1, z, \ldots), \quad &\text{if $x = 1$.}
	\end{cases}
\end{equation*}
We pinch the right end similarly.

We will \emph{knead} a finite sequence of positive integers by
\begin{itemize}
	\item Removing the leftmost entry, then
	\item Pinching both ends of what remains, then
	\item Placing the removed entry on the right end of the result
\end{itemize}
We note that in the second step, the ends can be pinched in either order and yield the same result. 

Let us clarify what happens with sequences of length $1$ or $2$.  We make the convention that the sequence $(1)$ and the empty sequence are pinched by doing nothing.  Then kneading fixes the sequences of length $1$, and kneading the sequence $(a,b)$ will give the sequence $(1,b-2,1,a)$ when $b \geq 3$ and $(b,a)$ when $b=1$ or $2$.

Kneading is an invertible process -- simply remove the rightmost entry, pinch both ends of what remains, then place the removed entry on the left end of the result.  It should also be apparent that kneading preserves the sum of the entries of a sequence.  Thus, kneading permutes the finitely many sequences with a given sum, and each sequence lies in a finite \emph{kneading cycle}.

\begin{kneadingexample}
Repeated kneading of the sequence $(2,2,3,6)$ gives:
\begin{align*}
	(2,2,3,6) &\mapsto (1,1,3,5,1,2) \mapsto (4,5,1,1,1,1) \mapsto (1,4,1,1,2,4) \mapsto
	(1,3,1,1,2,3,1,1) \notag\\
	&\mapsto (1,2,1,1,2,3,2,1) \mapsto (1,1,1,1,2,3,3,1) \mapsto (2,1,2,3,4,1) \mapsto (3,3,5,2)\\
	&\mapsto (1,2,5,1,1,3)  \mapsto (1,1,5,1,1,2,1,1) \mapsto (6,1,1,2,2,1) \mapsto (2,2,3,6).
\end{align*}
\end{kneadingexample}

Aside from the sum of the entries, kneading preserves two other sequence invariants -- the parity of the length, and a more subtle invariant we call the ``alternant'', which is a positive integer built from continued fractions.  

In our main result, Theorem \ref{T:main}, we provide a one-to-one correspondence between the set of sequences with given alternant $a$ and length parity $s = 0$ or $1$ and the set of Zagier-reduced indefinite binary quadratic forms of discriminant $a^2 + (-1)^s \cdot 4$, and show that kneading corresponds to Zagier reduction of the corresponding forms.  It follows that the number of sequences with alternant $a$ and parity $s$ is finite and that the number of cycles of such sequences is an (imprimitive) class number, the number of classes of quadratic forms of discrimint $a^2 + (-1)^s \cdot 4$, both primitive and imprimitive.

In the final section, we return to study the sum of the entries of sequences.  This invariant is not obvious when viewed from the perspective of quadratic forms.  We do some computations and make some observations and conjectures relating the sum invariant for a cycle of Zagier-reduced forms to the length of the cycle and the class group structure.  The number of forms in cycles under Gauss reduction was studied by Lachaud \cite{gL1984} \cite{gL1987}, and is called the \emph{caliber} of the cycle. To the author's knowledge, no similar study has been made of the caliber of cycles under Zagier reduction.

For the curious, kneading came about while the author was developing a generalization to indefinite forms of the Hardy-Muskat-Williams algorithm \cite{HMW1990} for representing an integer by a positive-definite form.  The correspondence between sequences and Zagier-reduced forms was discovered first.  Once the correspondence was discovered, it became a natural question to see what operation on sequences corresponds to reduction of forms, and the kneading pattern was then spotted.

\section{The correspondence}\label{S:correspondence}

A binary quadratic form is, for us, a polynomial $Ax^2 + Bxy + Cy^2$ in indeterminates $x$ and $y$ with integer coefficients.   We will now refer to them simply as ``forms''.  The question of which integers are obtained by inputting integers into a given form has motivated a tremendous amount of mathematics.  Famous results include Fermat's two squares theorem that the prime numbers represented by the form $x^2 + y^2$ are those congruent to $1$ modulo $4$ and the fact that for each nonsquare number $D > 0$, the Pell equation $x^2 - D y^2 = 1$ has a solution with $y \neq 0$.  

A general study of forms begins with the notion of Lagrange and Gauss of \emph{equivalent forms}.  Two forms are equivalent if one is transformed into the other by acting upon it with a $2 \times 2$ matrix with integer coefficients and determinant $1$, that is, a matrix in the group $\SL_2 (\mathbb{Z})$.  Specifically, a matrix $M = \begin{pmatrix} \alpha & \beta \\ \gamma & \delta \end{pmatrix}$ in $\SL_2 (\mathbb{Z})$ transforms the form $f(x,y)$ into the form
\begin{equation*}
	 f(\alpha x + \beta y, \gamma x + \delta y).
\end{equation*}
This gives a right action of $\SL_2 (\mathbb{Z})$ on the set of all forms.  

The discriminant of the form $Ax^2 + Bxy + Cy^2$ is $D = B^2-4AC$.  The $\SL_2(\mathbb{Z})$-action preserves discriminants, so the forms with fixed discriminant $D$ split into classes of equivalent forms. The first major theorem in the theory of binary quadratic forms is that the number of equivalence classes with given discriminant is finite.  The number of classes of forms all of which have relatively prime coefficients (\emph{primitive} forms) is a {class number}, a central notion of algebraic number theory.

A reduction algorithm is a standard tool for determining when two forms are equivalent.  Reduction is more complicated in the case of interest to us, when $D > 0$.  In fact, there are competing notions of reduction in this case.  We shall use Zagier reduction \cite{dZ1981} rather than the more common reduction of Lagrange and Gauss.

Zagier declares a form $f =Ax^2 + Bxy + Cy^2$ to be \emph{reduced} if
\begin{equation*}
	A > 0, \qquad C > 0, \qquad B > A+C.
\end{equation*}
To perform a reduction step on $f$, we 
\begin{enumerate}
	\item Compute the ``reducing number'', determined as the unique integer $n$ satisfying
\begin{equation*}
	n-1 < \frac{B+\sqrt{D}}{2A} < n,
\end{equation*}
	in which $D$ is the discriminant of $f$, then
	\item Act on the form $f$ with the matrix
\begin{equation*}
	\begin{bmatrix} n & 1 \\ -1 & 0 \end{bmatrix}.
\end{equation*}
\end{enumerate}
Zagier reduction is iteration of reduction steps.  

Because for a reduced form $D \geq D - (B-2A)^2 = 4A(B-A-C) > 0$, we see that the reduced forms with given discriminant $D$ have bounded $A$.  The same inequalities then imply that $B$ must be bounded, hence C must be as well.  There are thus finitely many Zagier-reduced forms with given positive discriminant.  Zagier shows that every form will reach a reduced form after finitely many reduction steps, after which it will continue through a cycle of reduced forms.   He also shows that two reduced forms are equivalent if and only if each can be obtained from the other by reduction, that is, both must lie in the same cycle of forms.  Thus, every equivalence class contains reduced forms, and the number of cycles of reduced primitive forms of given discriminant is a class number.

To define alternants, we turn to continued fractions.  Every rational number $\tfrac{\alpha}{\beta} > 1$ can be expanded in two ways as a finite simple continued fraction
\begin{equation*}
	\frac{\alpha}{\beta} = q_1 + \cfrac{1}{q_2 + \cfrac{1}{\ddots \, + \cfrac{1}{q_l}}}.
\end{equation*}
with positive integer quotients $q_1, \ldots, q_l$.  (Switching between the two expansions is accomplished by pinching the right end of this sequence.)

We will denote the numerator of the continued fraction with sequence of quotients $q_1$, \ldots, $q_l$ by $[q_1, \ldots, q_l]$. When $q_1, \ldots, q_l$ are indeterminates, these numerators are polynomials called \emph{continuants} \cite{tM1878}.
\begin{definition}
The \emph{alternant} of a finite sequence of positive integers $\OR{q} = (q_1, \ldots, q_l)$ with $l \geq 3$ is the difference
\begin{equation*}
	\left[ \OR{q} \right]^{*} :=  \left[ q_1, \ldots, q_l \right] -  \left[ q_2, \ldots, q_{l-1} \right]
\end{equation*}
We define directly the alternant of $(q_1)$ to be $q_1$ and of $(q_1, q_2)$ to be $q_1 q_2$.
\end{definition}

We define the \emph{length parity} of a finite sequence to be $0$ if the number of terms in the sequence is even and $1$ if the number is odd.  All sequences in a kneading cycle have the same length parity.

For integers $a > 0$ and $s=0$ or $1$, excepting the cases $(a,s)=(1,1)$ and $(2,1)$, we define sets
\begin{align*}
	S_{a,s} &= \{ \, \text{sequences of positive integers with alternant $a$ and length parity $s$} \, \}\\
	Z_{a,s} &= \{ \, \text{Zagier-reduced forms of discriminant $a^2 + (-1)^s \cdot 4$} \, \}.
\end{align*}

We define a map $\psi_{a,s} \colon Z_{a,s} \rightarrow S_{a,s}$ as follows.  If $f = Ax^2+ Bxy + Cy^2$ is in $Z_{a,s}$, we first compute $z = (a+B)/2$ (an integer) and expand the rational number $\tfrac{z}{A}$ into the unique continued fraction with sequence of quotients of length parity $s$.  We set $\psi_{a,s} (f)$ to be this sequence of quotients.  It will be shown to have alternant $a$ in Section \ref{S:proofs}

We also define a map $\phi_{a,s} \colon S_{a,s} \rightarrow Z_{a,s}$.  We define $\phi_{a,s} ((q_1, \ldots, q_l ))$ to be the form
\begin{equation}\label{E:correspondingform}
	 \left[ q_2, \ldots, q_l \right] x^2 + \left( \left[ q_1, \ldots, q_l \right] + \left[ q_2, \ldots, q_{l-1} \right] \right) xy  + \left[ q_1, \ldots, q_{l-1} \right] y^2.
\end{equation}
(When $l=1$, we should interpret this as $\phi_{a,s} ((q_1)) = x^2 + q_1 xy + y^2$.) In Section \ref{S:proofs}, the discriminant of the form \eqref{E:correspondingform} will be computed as $a^2 + (-1)^s \cdot 4$, where $a$ is the alternant of $(q_1, \ldots, q_l)$.

Our main theorem states:
\begin{theorem}\label{T:main}
The maps $\psi_{a,s}$ and $\phi_{a,s}$ are inverses, and through them Zagier reduction of forms corresponds to kneading.
\end{theorem}

\begin{example}
Consider the form $f = 44x^2 + 114 xy +17 y^2$, which has discriminant $100^2 + 4$.  To compute the corresponding sequence (with $a=100$ and $s=0$), we compute $z = (114+100)/2 = 107$, then expand $\tfrac{107}{44}$ as a continued fraction with even length
\begin{equation*}
	\frac{107}{44} = 2 + \cfrac{1}{2 + \cfrac{1}{3 + \cfrac{1}{6}}}.
\end{equation*}
Thus, $\psi_{100,0}(f) = (2,2,3,6)$, the sequence from the kneading example in Section \ref{S:intro}.  We reduce $f$ by computing the integer $n$ for which $n-1 < \tfrac{114+\sqrt{10004}}{88} < n$, that is, $n=3$, and then act on $f$ by the matrix $\begin{pmatrix} 3 & 1 \\ -1 & 0 \end{pmatrix}$ to obtain the new form $f' = 71x^2+150xy+44y^2$. To find $\psi_{100,0} (f')$, we calculate $z=(150+100)/2 = 125$, then expand $\tfrac{125}{71}$ as a continued fraction to obtain the sequence $\psi(f') = (1,1,3,5,1,2)$, the result of kneading $(2,2,3,6)$.
\end{example}

Theorem \eqref{T:main} provides an efficient method for producing all sequences with given alternant and length parity from a known list of Zagier-reduced forms of a certain discriminant.  Alternatively, from a known list of sequences with given alternant, we can compute the entire list of corresponding Zagier-reduced forms.  For instance, it can be shown that the set of all sequences with even length and alternant 11 comprises the sequence $(2,1,1,2)$, which is fixed by kneading, and a single other kneading cycle:
\begin{align*}
	(1,11) &\mapsto (1,9,1,1) \mapsto (1,8,2,1) \mapsto (1,7,3,1) \mapsto
	(1,6,4,1) \mapsto (1,5,5,1) \\
	&\mapsto (1,4,6,1) \mapsto (1,3,7,1) \mapsto (1,2,8,1) \mapsto (1,1,9,1) \mapsto (11,1) \mapsto (1,11).
\end{align*}
From these and \eqref{E:correspondingform}, we obtain the entire list of Zagier-reduced forms of discriminant $125$.  Representing the form $Ax^2 + Bxy + Cy^2$ by $(A,B,C)$, they are the imprimitive form $(5,15,5)$ and those in the reduction cycle
\begin{align*}
	&(11,13,1) \mapsto (19,31,11) \mapsto (25,45,19) \mapsto (29,55,25) \mapsto
	(31,61,29) \mapsto (31,63,31)\\
	&\mapsto (29,61,31) \mapsto (25,55,29) \mapsto (19,45,25) \mapsto (11,31,19) \mapsto (1,13,11) \mapsto (11,13,1).
\end{align*}
Lemmermeyer \cite{fL2014} notes the middle coefficients of the forms in some cycles steadily increase until they reach a maximum and then steadily decrease until they return to the minimum.  This phenomenon, visible in the cycle of forms above, is illuminated by the clear pattern in the corresponding kneading cycle.  The kneading cycle corresponding to the form $x^2 + (a+2) xy + a y^2$ of discriminant $a^2+4$ always exhibits a similar pattern -- see Lemma \ref{L:principalcyclea}.

We note that kneading can be used to perform reduction on Zagier-reduced forms with arbitrary non-square discriminant $D > 0$.  To accomplish this, begin by solving the Pell equation $x^2 - D y^2 = 4$ for integers $x$ and $y$ with $y \neq 0$.  If $f$ is a Zagier-reduced form of discriminant $D$, then consider the form $yf$ obtained by multiplying all coefficients of $f$ by $y$.  A glance at the Zagier reduction algorithm reveals that $yf$ and $f$ both have the same reducing number. Hence, if $f'$ is the form obtained by reducing $f$, then $yf'$ is the form obtained by reducing $yf$.  Multiplication by $y$ thus gives a bijection from the reduction cycle of $f$ to that of $yf$.  On the other hand, the discriminant of $yf$ is $y^2 D = x^2-4$.  Thus, we may reduce $f$ by determining $yf$ and kneading the corresponding sequence.

\section{Proofs}\label{S:proofs}

To begin, we develop some properties of continued fractions and continuants (see, for instance, \cite{PS1992}).

Beginning with $[ \cdot ] = 1$ and $[ q_1 ] = q_1$, continuants satisfy the recurrences
\begin{equation}
\begin{aligned}\label{E:contrecursion}
	\left[ q_1, \ldots, q_l \right] &= q_1 \left[ q_2, \ldots, q_l \right] + \left[ q_3, \ldots, q_l \right] \, \text{ or}\\
	\left[ q_1, \ldots, q_l \right] &= q_l \left[ q_1, \ldots q_{l-1} \right] + \left[ q_1, \ldots, q_{l-2} \right].
\end{aligned}
\end{equation}
We adopt, for now, the first as our definition and later show that it gives the numerator of an appropriate continued fraction.  The equivalence with the second recurrence and all other properties we will need follow elegantly from the matrix identity
\begin{equation}\label{E:matrixcontinuant}
	\begin{bmatrix} q_1 & 1 \\ 1 & 0 \end{bmatrix} \begin{bmatrix} q_2 & 1 \\ 1 & 0 \end{bmatrix} \cdots \begin{bmatrix} q_l & 1 \\ 1 & 0 \end{bmatrix} = \begin{bmatrix} [q_1, \ldots, q_l] & [q_1, \ldots, q_{l-1} ] \\ [q_2, \ldots, q_l] & [q_2, \ldots q_{l-1} ] \end{bmatrix},
\end{equation}
which can be verified by induction using the first recursion \eqref{E:contrecursion}. 

Transposing both sides of \eqref{E:matrixcontinuant} reveals the surprising symmetry $\left[ q_1, \ldots, q_l \right] = \left[ q_l, \ldots, q_1 \right]$, from which follows the second recursion \eqref{E:contrecursion}.  Taking determinants in \eqref{E:matrixcontinuant} yields another useful identity
\begin{equation}\label{E:contdeterminant}
	\left[ q_1, \ldots, q_l \right] \left[ q_2, \ldots, q_{l-1} \right] - \left[ q_1, \ldots, q_{l-1} \right]. \left[ q_2, \ldots, q_l \right] = (-1)^{l}.
\end{equation}

We also note the simplifications
\begin{align}
	\left[ q_1, q_2, \ldots, q_i, 0, q_{i+1}, \ldots, q_l \right] &= \left[ q_1, q_2, \ldots, q_{i-1}, q_i + q_{i+1}, q_{i+2},  \ldots, q_l \right] \label{E:zerosimplify} \\
	\left[ 0, q_1, \ldots, q_l \right] &= \left[ q_2, \ldots, q_l \right] \label{E:endzerosimplify}\\
	\left[ 1, q_1, \ldots, q_l \right] &= \left[ q_1+1, \ldots, q_l \right]. \label{E:onesimplify}
\end{align}
The first follows from \eqref{E:matrixcontinuant} and the computation
\begin{equation*}
	\begin{bmatrix} q_i & 1 \\ 1 & 0 \end{bmatrix} \begin{bmatrix} 0 & 1 \\ 1 & 0 \end{bmatrix} \begin{bmatrix} q_{i+1} & 1 \\ 1 & 0 \end{bmatrix} = \begin{bmatrix} q_i + q_{i+1} & 1 \\ 1 & 0 \end{bmatrix},
\end{equation*}
and the others follow readily from the recursion \eqref{E:contrecursion}.

Now let us return to continued fractions.  We can prove inductively that
\begin{equation}\label{E:cfracnumerator}
	q_1 + \cfrac{1}{q_2 + \cfrac{1}{\ddots \, + \cfrac{1}{q_l}}} = \frac{\left[ q_1, \ldots, q_l \right]}{\left[ q_2, \ldots, q_{l} \right]}.
\end{equation}
We see from \eqref{E:contdeterminant} that this fraction is in lowest terms, so the continuant $\left[ q_1, \ldots, q_l \right]$ is the numerator when the continued fraction with partial quotients $q_1$, \ldots, $q_l$ is fully simplified. 

 Now we prove that for given integers $a > 0$ and $s = 0$ or $1$ with $(a,s) \neq (1,1)$ or $(2,1)$, in order:
 \renewcommand{\theenumi}{\roman{enumi}}
\begin{enumerate}
	\item If $(q_1, \ldots, q_l)$ has alternant $a$, then $\phi_{a,s} ((q_1, \ldots, q_l))$ is Zagier-reduced with discriminant $a^2 + (-1)^l \cdot 4$,
	\item $\psi_{a,s} \circ \phi_{a,s}$ is the identity map on $S_{a,s}$,
	\item If $f$ is a form of discriminant $a^2 + (-1)^s \cdot 4$, then $\psi_{a,s} (f)$ has alternant $a$ and length parity $s$,
	\item $\phi_{a,s} \circ \psi_{a,s}$ is the identity map on $Z_{a,s}$,
	\item Kneading corresponds to Zagier reduction of forms.
\end{enumerate}

\noindent \textbf{(i):} This is easily verified when $l=1$ or $2$, so let $l \geq 3$.  Let $(q_1, \ldots, q_l)$ be a sequence of positive integers with alternant $a$, and let $\phi_{a,s} ((q_1, \ldots, q_l)) = Ax^2 + Bxy + Cy^2$ be the form \eqref{E:correspondingform}.  To see that it is reduced, note that coefficients $A$ and $C$ are clearly positive, so we need only check that $B > A + C$.  Using \eqref{E:contrecursion}, we compute
\begin{align*}
	B - C &= \left( q_l -1 \right) \left[ q_1, \ldots, q_{l-1} \right] + \left[ q_1, \ldots, q_{l-2} \right] + \left[ q_2, \ldots, q_{l-1} \right]\\
	&>  (q_l - 1) \left[ q_2, \ldots, q_{l-1} \right] + \left[ q_2, \ldots, q_{l-2} \right] + \left[ q_2, \ldots, q_{l-1} \right]\\
	&= q_{l} \left[ q_2, \ldots, q_{l-1} \right] + \left[ q_2, \ldots, q_{l-2} \right]  = \left[ q_2, \ldots, q_{l} \right] = A.
\end{align*}
For the discriminant we compute, using \eqref{E:contdeterminant} and \eqref{E:contrecursion}, 
\begin{align*}
	&\left( \left[ q_1, \ldots, q_{l} \right] + \left[ q_2, \ldots, q_{l-1} \right] \right)^2 - 4 \left[ q_2, \ldots, q_{l} \right] \left[ q_1, \ldots, q_{l-1} \right] \\
	&= \left[ q_1, \ldots, q_{l} \right]^2 + \left[ q_2, \ldots, q_{l-1} \right]^2 - 2 \left[ q_2, \ldots, q_{l} \right] \left[ q_1, \ldots, q_{l-1} \right] + (-1)^{l} \cdot 2 \\
	&= \left( q_1 \left[ q_2, \ldots, q_{l} \right] \right)^2 + 2 q_1 \left[ q_2, \ldots, q_{l} \right] \left[ q_3, \ldots, q_{l} \right] + \left[ q_3, \ldots, q_{l} \right]^2 + \left[ q_2, \ldots, q_{l-1} \right]^2 \\
	&\qquad - 2q_1 \left[ q_2, \ldots, q_{l} \right] \left[ q_2, \ldots, q_{l-1} \right] - 2 \left[ q_2, \ldots, q_{l} \right] \left[ q_3, \ldots, q_{l-1} \right]+ (-1)^{l} \cdot 2\\
	&= \left( q_1 \left[ q_2, \ldots, q_{l} \right] - \left[ q_2, \ldots, q_{l-1} \right] + \left[ q_3, \ldots, q_{l} \right] \right)^2 + (-1)^{l } \cdot 2 \\
	&\qquad - \left( 2 \left[ q_2, \ldots, q_{l} \right]\left[ q_3, \ldots, q_{l-1} \right] - 2 \left[ q_2, \ldots, q_{l-1} \right] \left[ q_3, \ldots, q_{l} \right] \right)\\
	&= a^2 + (-1)^{l} \cdot 4. \qedhere
\end{align*}

\noindent \textbf{(ii):}  The definition of alternants and \eqref{E:correspondingform}  show that the sequence\\ $\psi_{a,s} \left( \phi_{a,s} ((q_1, \ldots, q_l)) \right)$ is obtained by expanding in a continued fraction the rational number with denominator $\left[ q_2, \ldots, q_l \right]$ and numerator $\left[ q_1, \ldots, q_l \right]$.  From \eqref{E:cfracnumerator} and the well-known uniqueness of continued fraction expansions, this sequence is $( q_1, \ldots, q_l )$ (which has the right length parity).   \\

\noindent \textbf{(iii):} If $a=1$ or $2$, so by our assumption $s=0$, then it is not hard to show that the only Zagier-reduced forms of discriminant $a^2 + (-1)^s \cdot 4$ are $x^2 + 3xy + y^2$, $x^2 + 4xy + 2y^2$, and $2x^2 + 4xy + y^2$.   Applying $\psi_{1,0}$ to the first and $\psi_{2,0}$ to the other two gives the sequences $(1,1)$, $(2,1)$, and $(1,2)$ of alternants $1$, $2$, and $2$ respectively.  Thus (iii) holds in these cases.

Now choose a Zagier-reduced form $f = A x^2 + Bxy + Cy^2$ of discriminant $D = a^2 + (-1)^s \cdot 4$ with $a > 2$, $s=0$ or $1$, and $D > 0$.  By design, the length parity of $\psi_{a,s} (f)$ is $s$, so we need only worry about the alternant.

First, $B$ and $D$ have the same parity, hence $a$ and $B$ do.  The positive integer $z = (a+B)/2$ is thus a divisor of  
\begin{equation*}
\frac{B^2 - a^2}{4} =  AC + \frac{D}{4} -\frac{a^2}{4} = AC + (-1)^{s}.
\end{equation*}
Thus, $A$ is relatively prime to $z$ and $AC \equiv (-1)^{s+1} \pmod{z}$.

Note as well that $a^2 + (-1)^s \cdot 4 = B^2 - 4AC > (A-C)^2$ since $f$ is reduced. Then $a > |A-C|$ since $a > 2$, so $a + A > C$. Hence, using again that $f$ is reduced, we have $z > (a+A+C)/2 > C$.  Since also $a+ C > A$, we also have $z > A$.

 Expand $z/A$ as a simple continued fraction with sequence of quotients $(q_1, \ldots q_l)$, and choose the length so that $l$ and $s$ have the same parity.  From \eqref{E:cfracnumerator}, we have $z = \left[ q_1, \ldots, q_l \right]$ and $A = \left[ q_2, \ldots, q_{l} \right]$.  Since $AC \equiv (-1)^{s+1} \pmod{z}$, we also have from \eqref{E:contdeterminant} the congruence $C \equiv \left[ q_1, \ldots, q_{l-1} \right] \pmod{z}$.  Since $0 < C < z$, it follows that $C = \left[ q_1, \ldots, q_{l-1} \right]$.

From \eqref{E:contdeterminant}, we have
\begin{align*}
	\left[ q_2, \ldots, q_{l-1} \right] = 2 \frac{AC + (-1)^l}{a+B} &= \frac{B^2 - (a^2+(-1)^l \cdot 4) + (-1)^l \cdot 4}{2(a+B)}\\
	&= \frac{B-a}{2}.
\end{align*}
 
 Thus, the alternant of $\psi_{a,s} (f)$ is 
 \begin{equation*}
 	\left[ q_1, \ldots, q_l \right] - \left[ q_2, \ldots, q_{l-1} \right] = \frac{B+a}{2} - \frac{B-a}{2} = a.
 \end{equation*}
 
 \noindent \textbf{(iv):} Let $f=A x^2 + Bxy + C y^2$ be as in (iii).  The verification of (iii) shows at least that the form $\phi_{a,s} \circ \psi_{a,s} (f)$ is $A x^2 + B' xy + Cy^2$ for some positive integer $B'$.  Also, (i) and (iii) show that  $B'$ satisfies $B'^2 - 4AC = a^2 + (-1)^s \cdot 4$.  But $B$ is the unique such positive integer, thus $\phi_{a,s} \circ \psi_{a,s} (f) = f$.  \\
 
\noindent \textbf{(v):}  Suppose that $(q_1, \ldots, q_l)$ is a sequence with alternant $a$ and length parity $s$.  The reducing number for Zagier reduction of $\phi_{a,s} ((q_1, \ldots, q_l))$ is
\begin{equation}\label{E:reducingnumber}
	\left\lceil \frac{\left[ q_1, \ldots, q_{l} \right] + \left[ q_2, \ldots, q_{l-1} \right] + \sqrt{D}}{2 \left[ q_2, \ldots, q_{l} \right]} \right\rceil,
\end{equation}
 where $D = a^2 + (-1)^{s} \cdot 4$ is the discriminant.  When $l=1$, so $q_1 > 2$, the number inside the ceiling is $\tfrac{q_1 + \sqrt{q_1^2 - 4}}{2}$, making the value of the ceiling $q_1$.  When $l=2$, the reducing number is 
 \begin{equation*}
	\left\lceil \frac{q_1 q_2+2+\sqrt{(q_1 q_2)^2+4}}{2q_2} \right\rceil .
\end{equation*}
A little algebra shows that this ceiling is $q_1+1$ when $q_2 \geq 2$ and $q_1+2$ when $q_2 = 1$.  A direct check shows that in these cases, reducing the form using the appropriate matrix corresponds to kneading the corresponding sequence. \\

Otherwise, for $l \geq 3$ the term $\sqrt{D}$ in the numerator of \eqref{E:reducingnumber} is approximately $a$, so the whole numerator is approximately 
 \begin{equation*}
 	2 \left[ q_1, \ldots, q_{l} \right] = 2q_1 \left[ q_2, \ldots, q_{l} \right] + 2 \left[ q_3, \ldots, q_{l} \right],
\end{equation*}
making, at least approximately, the expression in the ceiling in \eqref{E:reducingnumber} between $q_1$ and $q_1+1$.  Some algebra shows in fact that the exact quotient is between $q_1$ and $q_1+1$, so the reducing number is always $q_1+1$ in this case (the quotient is equal to $1$ when $l=3$ and $q_1=q_3=1$, but this falls under the case $(a,s)=(2,1)$ that we are excluding). Acting on $\phi_{a,s} ((q_1, \ldots, q_l))$ by the reduction matrix $\begin{pmatrix} q_1 + 1 & 1 \\ -1 & 0 \end{pmatrix}$, the theorem follows by checking the formulas (\eqref{E:zerosimplify} and \eqref{E:endzerosimplify} show these are appropriate even when $q_2$ or $q_l$ is $1$):
\begin{align}
	&\left[ q_2 - 1, q_3, \ldots, q_{l-1}, q_{l}-1, 1, q_1 \right] \label{E:longcoefficient} \\
	&= (q_1+1)^2 \left[ q_2, \ldots, q_{l} \right] - (q_1+1) \left( \left[ q_1, \ldots, q_{l} \right] + \left[ q_2, \ldots, q_{l-1} \right] \right) + \left[ q_1, \ldots, q_{l-1} \right]\notag\\
	&\left[ 1, q_2 - 1, q_3, \ldots, q_{l-1}, q_{l} - 1, 1, q_1 \right] + \left[ q_2 - 1, q_3, \ldots, q_{l-1}, q_{l} - 1, 1 \right] \label{E:middlecoefficient}\\
	&= (2q_1+2) \left[ q_2, \ldots, q_{l} \right] - \left( \left[ q_1, \ldots, q_{l} \right] + \left[ q_2, \ldots, q_{l-1} \right] \right)\notag\\
	&\left[ 1, q_2 - 1, q_3, \ldots, q_{l-1}, q_{l} - 1, 1 \right] = \left[ q_2, \ldots, q_{l} \right]. \label{E:shortcoefficient}
\end{align}
First separating off a $q_1$ from the second and fourth continuants on the right side of \eqref{E:longcoefficient}and then repeatedly applying \eqref{E:contrecursion} simplifies it to
\begin{align*}
	&(q_1+1) \left( \left[ q_2, \ldots, q_{l} \right] - \left[ q_3, \ldots, q_{l} \right] \right) - \left[ q_2, \ldots, q_{l-1} \right] + \left[ q_3, \ldots, q_{l-1} \right]\\
	&= (q_2 - 1)(q_1+1) \left[ q_3, \ldots, q_{l} \right] + (q_1+1) \left[ q_4, \ldots, q_{l} \right]\\
	&\qquad - (q_2 - 1) \left[ q_3, \ldots, q_{l-1} \right] - \left[ q_4, \ldots, q_{l-1} \right]\\
	&= (q_1+1) \left[ q_2 - 1, q_3, \ldots, q_{l} \right] - \left[ q_2 - 1, q_3, \ldots, q_{l-1} \right]\\
	&= (q_1+1) \left[ q_2 - 1, q_3, \ldots, q_{l-1}, q_{l} - 1 \right] +q_1 \left[ q_2 - 1, q_3, \ldots, q_{l-1}  \right]\\
	&= q_1 \left[ q_2 - 1, q_3, \ldots, q_{l-1}, q_{l} - 1, 1 \right] + \left[ q_2 - 1, q_3, \ldots, q_{l-1}, q_{l} - 1 \right]\\
	&= \left[ q_2 - 1, q_3, \ldots, q_{l-1}, q_{l} - 1, 1, q_1 \right].
\end{align*}
With this, \eqref{E:longcoefficient} is verified.  The verification of \eqref{E:middlecoefficient} is similar, but shorter, after first simplifying the left side to
\begin{equation*}
	 q_1 \left[ q_2, \ldots, q_{l} \right] + \left[ q_2, \ldots, q_{l-1}, q_{l} - 1 \right] + \left[ q_2 - 1, q_3, \ldots, q_{l} \right].
\end{equation*}
Equation \eqref{E:shortcoefficient} follows immediately from \eqref{E:onesimplify}.

\section{The sum invariant}

The standard class invariants of an indefinite form are the discriminant and the greatest common divisor of the coefficients. Theorem \ref{T:main} identifies a new invariant for forms with discriminant of the form $a^2 + (-1)^s \cdot 4$: the sum of the corresponding sequences.  We will call this number the ``sum of a form''. Because it is a class invariant, we also write the ``sum of a class'' to mean the sum of any representative of the class.       

The classes of primitive forms of common discriminant have the structure of a finite abelian group with the group operation called \emph{Gauss composition}.  We conjecture in this section a link between a reduced form's sum, the length of its cycle, and in Conjecture \ref{C:composition}, the group structure.    The number of Gauss-reduced forms in a cycle was studied by Lachaud \cite{gL1984} \cite{gL1987} and is called the \emph{caliber} of the cycle.  In what follows, we shall use ``caliber'' to refer to the length of a cycle of Zagier-reduced forms.

The conjectures we present were formulated after examining data.   It is certainly much easier to compute data about the sum invariant from the point of view of kneading cycles.  Kneading is simple to implement with standard list manipulation algorithms.  The sequences with sum $n$ are just the \emph{compositions} (i.e., ordered partitions) of $n$.  The author used \emph{Mathematica} \cite{Wolfram} to enumerate these compositions and place them into kneading cycles. Conjecture \ref{C:composition} below was discovered with the aid of Matthews' tools at \texttt{www.numbertheory.org}.

Forms with the same sum invariant are distributed across different discriminants.  One naturally wonders to what extent the cycle of a form is determined by its discriminant, the greatest common divisor of its coefficients, and the sum invariant.  Some cycles are not uniquely determined by these values.  For instance, classes that are inverses of each other under Gauss composition share all invariants.  Conjecture \ref{C:composition} below implies that other pairs of classes share all invariants.  Outside of these cases, it is hard to find counterexamples with small sum.  But, for instance, the primitive forms $5x^2+30xy+11y^2$ and $7x^2+34xy+17y^2$ each have discriminant 680 and sum 9 but are not equivalent.  They also are not in inverse classes and do not meet the hypotheses of Conjecture \ref{C:composition}.

\subsection{Forms with discriminant of the form $a^2+4$}

The identity class for the Gauss composition operation is called the \emph{principal class}. The cycle of Zagier-reduced forms in this class is the \emph{principal cycle}. When the discriminant has the form $a^2 \pm 4$, we call also call the corresponding kneading cycle the \emph{principal kneading cycle} of alternant $a$ and length parity $0$. 

The principal class of discriminant $a^2+4$ can be shown to be the class of the form $x^2 + (a+2)xy + ay^2$.

\begin{lemma}\label{L:principalcyclea}
For $a \geq 3$, the principal kneading cycle of alternant $a$ and length parity $0$ has the pattern
\begin{equation*}
	(a, 1) \mapsto (1, a) \mapsto (1, a-2, 1, 1) \cdots  \mapsto (1, a-k, k-1, 1) \mapsto \cdots \mapsto (1, 1, a-2, 1) \mapsto (a,1).
\end{equation*}
In particular, the associated sum invariant is $a+1$.
\end{lemma}

\begin{proof}
The principal kneading cycle contains the sequence $(a,1)$ corresponding to the form $x^2 + (a+2) xy + ay^2$.  The pattern is then easily checked.
\end{proof}

\begin{lemma}
If $n$ is the sum of a class with discriminant of the form $a^2 + 4$, then $n \leq a+1$, with equality if and only if the class is principal.
\end{lemma}

\begin{proof}
Suppose $\OR{q} = (q_1, \ldots, q_l)$ is the sequence corresponding to a form in the class, so $l$ is even.  If $l = 2$, then $n=q_1 + q_2 \leq q_1 q_2 + 1 = a+1$, with equality only for the sequences $(a,1)$ and $(1,a)$. Otherwise, using \eqref{E:contrecursion}, we have
\begin{equation}
\begin{aligned}\label{E:alternantbound}
	a &= \left[ q_1, \ldots, q_l \right] - \left[ q_2, \ldots, q_{l-1} \right] \\
	 &= q_1 \left[ q_2, \ldots,  q_l \right] + \left[ q_3, \ldots, q_l \right] - \left[ q_2, \ldots, q_{l-1} \right] \\
	 &= (q_1 q_l-1) \left[ q_2, \ldots, q_{l-1} \right] + q_1 \left[ q_2, \ldots, q_{l-2} \right] + q_l \left[ q_3, \ldots, q_{l-1} \right] + \left[ q_3, \ldots, q_{l-2} \right].
\end{aligned}
\end{equation}

It can be checked by induction that for any sequence of positive integers $(c_1, \ldots, c_r)$, we have $\left[ c_1, \ldots, c_r \right] \geq c_1 + \cdots + c_r$.  
Thus, unless $q_1=q_l=1$, we obtain from \eqref{E:alternantbound} the inequality $a > q_1 + \cdots + q_l = n$.  If $q_1 = q_l = 1$ and $l \geq 6$, then \eqref{E:alternantbound} gives
\begin{equation*}
	a = \left[ q_2, \ldots, q_{l-2} \right] + \left[ q_3, \ldots, q_{l-1} \right] + \left[ q_3, \ldots, q_{l-2} \right] \geq q_2 + 3 (q_3 + \cdots + q_{l-2}) + q_{l-1} > n.
\end{equation*}
Otherwise, if $q_1 = q_l = 1$ and $l = 4$, then
\begin{equation*}
	a = \left[ 1, q_2, q_3, 1 \right] - \left[ q_2, q_3 \right] =  q_2 + q_3 + 1 = n-1.
\end{equation*}
The lemma now follows from Lemma \ref{L:principalcyclea}.
\end{proof}

 We note that these lemmas provide a fast test to decide if a form with discriminant $a^2+4$ is in the principal class:  use the Euclidean algorithm to compute the corresponding sequence.  If the algorithm takes more than three steps, the form is not principal.  Otherwise, it is principal only if the sum of the quotients that appeared is $a+1$.

Computations quickly reveal a striking regularity in the cycles of forms with discriminant of the form $a^2+4$ and fixed sum $n$. The evidence seems strong enough to formulate these regularities as conjectures.
\begin{conjecture}\label{C:divisor}
Consider a cycle of reduced forms of caliber $l$ with discriminant of the form $a^2+4$ and sum $n$.  Then
\begin{enumerate}
	\item $n = (2r+1)l + 1$ for some integer $r \geq 0$,
	\item if $r > 0$, then the cycle consists of imprimitive forms (i.e., the gcd of the coefficients of each form is $> 1$).
\end{enumerate}
\end{conjecture}

The number of cycles with fixed caliber and sum also shows a pattern:
\begin{conjecture}\label{C:formula}
Consider the set of all cycles of Zagier-reduced forms with discriminants of the form $a^2 +4$ and caliber $l$.  The number of such cycles with fixed sum $n$ satisfying (i) in Conjecture \ref{C:divisor} is
	\begin{equation*}
		\frac{1}{2l} \sum_{\substack{ d \mid l\\d  \, \, \mathrm{odd}}}  \mu(d) 2^{l/d},
	\end{equation*}
in which $\mu$ is the M\"{o}bius function.  In particular, the number is independent of the sum.
\end{conjecture}
For instance, cycles of caliber $4$ only occur with sums of the form $8r+5$, and there are exactly two caliber-$4$ cycles with each such sum.  Cycles of caliber $7$ only occur with sums of the form $14r+8$, and there are nine caliber-$7$ cycles with each such sum.

Both of the above conjectures were checked for $2 \leq n \leq 31$.   The formula in Conjecture \ref{C:formula} was produced through a search of the On-Line Encyclopedia of Integer Sequences \cite{OEIS}.

 It is well known that the number of $k$-term compositions of a positive integer $n$  is the binomial coefficient $\binom{n-1}{k-1}$.  It follows readily that there are $2^{n-2}$ Zagier-reduced forms with discriminant of the form $a^2+4$ and sum $n$, and an equal number with discriminant of the form $a^2-4$ and sum $n$.  On the other hand, Conjectures \ref{C:divisor} and \ref{C:formula} imply that the number of Zagier-reduced forms with discriminant of the form $a^2+4$ and sum $n$ is 
\begin{equation*}
	\frac{1}{2} \sum_{ \substack{ k \mid n-1 \\  k \text{ odd}}} \sum_{ \substack{ d \mid \tfrac{n-1}{k}\\ d \text{ odd}}} \mu(d) 2^{\frac{n-1}{kd}} = \frac{1}{2} \sum_{ \substack{ b \mid n-1 \\ b \text{ odd}}} \left( 2^{\frac{n-1}{b}} \sum_{c \mid b} \mu(c) \right) = 2^{n-2},
\end{equation*} 
in agreement with the known value.

Let us call a cycle of forms of discriminant $a^2+4$ with sum $n$ but caliber less than $n - 1$ a \emph{short cycle}.  Conjecture \ref{C:divisor} implies that such cycles consist of imprimitive forms.  We provide a table specifying the different short cycles with sum between $1$ and $29$.  Each row corresponds to a cycle, listed in increasing order of sum. The second column gives the cycle caliber.  The third and fourth columns together specify a form $f$ in the corresponding cycle. The third column contains $d$, the gcd of the coefficients of $f$, while the fourth column contains the triple $(A,B,C)$ where $Ax^2 + Bxy + Cy^2$ is the primitive form obtained by dividing each coefficient of $f$ by $d$.  We have chosen the representative form in each case to be that with the smallest middle coefficient $B$.

\begin{tabular}{c | c | c | c }
	$\mathrm{Sum} - 1$ & Caliber & $d$ & Form\\
	\hline
	3 & 1 & 2 & (1,3,1)\\
	5 & 1 & 5 & (1,3,1)\\
	6 & 2 & 5 & (1,4,2)\\
	7 & 1 & 13 & (1,3,1)\\
	9 & 1 & 34 & (1,3,1)\\
	9 & 3 & 10 & (1,5,3)\\
	10 & 2 & 29 & (1,4,2)\\
	11 & 1 & 89 & (1,3,1)\\
	12 & 4 & 17 & (1,6,4)\\
	12 & 4 & 37 & (2,8,3)\\
	13 & 1 & 233 & (1,3,1)\\
	14 & 2 & 169 & (1,4,2)\\
	15 & 1 & 610 & (1,3,1)\\
	15 & 3 & 109 & (1,5,3)\\
	15 & 5 & 130 & (1,5,2)\\
	15 & 5 & 26 & (1,7,5)\\
	15 & 5 & 82 & (3,11,3)\\
	17 & 1 & 1597 & (1,3,1)\\
	18 & 2 & 985 & (1,4,2)\\
	18 & 6 & 17 & (1,8,6)\\
	18 & 6 & 101 & (2,12,5)\\
	18 & 6 & 145 & (3,14,4)\\
	18 & 6 & 145 & (4,14,3)\\
	18 & 6 & 257 & (5,20,7)\\
	19 & 1 & 4181 & (1,3,1)\\
	20 & 4 & 305 & (1,6,4)\\
	20 & 4 & 1405 & (2,8,3)\\
	21 & 1 & 10946 & (1,3,1)\\
	21 & 3 & 1189 & (1,5,3)\\
	21 & 7 & 290 & (1,7,3)\\
	21 & 7 & 514 & (2,9,2)\\
	21 & 7 & 50 & (1,9,7)\\
	21 & 7 & 1154 & (4,15,5)\\
	21 & 7 & 226 & (3,17,5)\\
	21 & 7 & 226 & (5,17,3)\\
	21 & 7 & 442 & (5,25,9)\\
	21 & 7 & 362 & (5,25,13)\\
	21 & 7 & 530 & (7,27,7)\\
	22 & 2 & 5741 & (1,4,2)\\
	23 & 1 & 28657 & (1,3,1)\\
	24 & 8 & 65 & (1,10,8)\\
	24 & 8 & 197 & (2,16,7)\\
	24 & 8 & 325 & (3,20,6)\\
	24 & 8 & 325 & (6,20,3)\\
	24 & 8 & 401 & (4,22,5)\\
	24 & 8 & 401 & (5,22,4)\\
	24 & 8 & 577 & (5,30,16)\\
	24 & 8 & 677 & (5,30,11)\\
\end{tabular} \, \, \begin{tabular}{c | c | c | c }
	$\mathrm{Sum} - 1$ & Caliber & $d$ & Form\\
	\hline
	24 & 8 & 901 & (7,34,9)\\
	24 & 8 & 901 & (9,34,7)\\
	24 & 8 & 677 & (7,34,17)\\
	24 & 8 & 1025 & (8,38,13)\\
	24 & 8 & 1025 & (13,38,8)\\
	24 & 8 & 1157 & (10,40,11)\\
	24 & 8 & 1297 & (12,46,17)\\
	24 & 8 & 1765 & (13,52,18)\\
	25 & 1 & 75025 & (1,3,1)\\
	25 & 5 & 8578 & (1,5,2)\\
	25 & 5 & 701 & (1,7,5)\\
	25 & 5 & 6805 & (3,11,3)\\
	26 & 2 & 33461 & (1,4,2)\\
	27 & 1 & 196418 & (1,3,1)\\
	27 & 3 & 12970 & (1,5,3)\\
	27 & 9 & 514 & (1,9,4)\\
	27 & 9 & 82 & (1,11,9)\\
	27 & 9 & 1154 & (2,13,3)\\
	27 & 9 & 1154 & (3,13,2)\\
	27 & 9 & 442 & (3,23,7)\\
	27 & 9 & 442 & (7,23,3)\\
	27 & 9 & 3202 & (4,23,8)\\
	27 & 9 & 3202 & (8,23,4)\\
	27 & 9 & 3874 & (5,25,7)\\
	27 & 9 & 626 & (5,27,5)\\
	27 & 9 & 4610 & (6,29,11)\\
	27 & 9 & 4610 & (11,29,6)\\
	27 & 9 & 962 & (5,35,13)\\
	27 & 9 & 842 & (5,35,19)\\
	27 & 9 & 7202 & (9,35,9)\\
	27 & 9 & 1370 & (7,41,11)\\
	27 & 9 & 1370 & (11,41,7)\\
	27 & 9 & 1090 & (7,41,21)\\
	27 & 9 & 1090 & (21,41,7)\\
	27 & 9 &  1522 & (9,43,9)\\
	27 & 9 & 2026 & (11,51,13)\\
	27 & 9 & 2026 & (13,51,11)\\
	27 & 9 & 1850 & (11,51,17)\\
	27 & 9 & 2810 & (17,63,17)\\
	27 & 9 & 2402 & (17,63,23)\\
	27 & 9 & 3026 & (13,65,23)\\
	27 & 9 & 3722 & (19,75,25)\\
	27 & 9 & 3970 & (21,79,27)\\
	28 & 4 & 5473 & (1,6,4)\\
	28 & 4 & 53353 & (2,8,3)\\
	29 & 1 & 514229 & (1,3,1) \\
	& & &\\
	& & &
\end{tabular}

\subsection{Forms with discriminant of the form $a^2-4$}

The principal cycle with discriminant of the form $a^2-4$ consists of the single form $x^2 + axy + y^2$ of sum $a$.  The corresponding kneading cycle consists of the single sequence $(a)$.  In addition, for $a \geq 4$, the form 
\begin{equation}\label{E:specialform}
	(a-2) x^2 + (3a-6) xy + (2a-5) y^2
\end{equation}
of discriminant $a^2-4$ always lies in a cycle of caliber $a-2$.  A simple way to see this is to note that the corresponding sequence is $(2,a-3,1)$ and that the kneading cycle for this sequence follows a simple pattern
\begin{align*}
	(2,a-3,1) &\mapsto (1,a-3,2) \mapsto (1,a-4,1,1,1) \mapsto (1,a-5,1,2,1) \mapsto \cdots \\	&\mapsto (1,1,1,a-4,1) \mapsto (2,a-3,1).
\end{align*}
Let us denote by $\mathfrak{c}$ the class of forms corresponding to this cycle.  Because it contains both the form \eqref{E:specialform} and its opposite (obtained by switching the coefficients of $x^2$ and $y^2$), it has order two in the class group.

\begin{lemma}
If $n$ is the sum of a class with discriminant of the form $a^2 - 4 > 0$, then $n \leq a$, with equality if and only if the class is principal or equal to $\mathfrak{c}$.
\end{lemma}

\begin{proof}
Suppose $\OR{q} = (q_1, \ldots, q_l)$ is the sequence corresponding to a form in the class, so $l$ is odd.  If $l = 1$, then the sequence is $(a)$, in the principal kneading cycle, and the lemma is immediate. Otherwise, if $l=3$, then $a = (q_1 q_3 - 1) q_2 + q_1 + q_3$.  We cannot have $q_1 q_3 = 1$, since in this case $a^2 - 4 = 0$, which we have excluded from consideration.  Thus, $a \geq q_1 + q_2 + q_3 = n$, with equality only when $q_1 q_3 = 2$.

Finally, if $l \geq 5$, then the equalities \eqref{E:alternantbound} are valid and we have
\begin{align*}
	a &\geq q_1 \left[ q_2, \ldots, q_{l-2} \right] + q_l \left[ q_3, \ldots, q_{l-1} \right] + (q_3 + \cdots + q_{l-2}) \\
	&\geq q_1 (q_2 + 1) + (q_3 + \cdots + q_{l-2}) + q_l (q_{l-1} + 1)\\
	&\geq q_1 + \cdots + q_{l} = n,
\end{align*}
with equality only occurring with $l=5$ and $q_1 = q_3 = q_5 = 1$.

The lemma now follows from the description of the kneading cycle corresponding to the class $\mathfrak{c}$ given above.
\end{proof}

The calibers of cycles with discriminant of the form $a^2-4$ and fixed sum exhibit large variation.  However, these cycles pair off in such a way that the sum of the calibers of the cycles in each pair exhibits regularity.  Conjecture \ref{C:composition} states that the pairing arises from the class group structure.

\begin{conjecture}\label{C:composition}
For fixed $a \geq 3$, let $\mathfrak{c}$ be the class of discriminant $a^2-4$ described above.  Let $\mathfrak{c}_1$ and $\mathfrak{c}_2$ be primitive classes whose Gauss composition equals $\mathfrak{c}$, and let $l_1$ and $l_2$ be the calibers of their cycles of reduced forms.   Then 
\begin{enumerate}
	\item $\mathfrak{c}_1$ and $\mathfrak{c}_2$ have the same sum $n \leq a$,
	\item $l_1 + l_2 = n-1$.
\end{enumerate}
\end{conjecture}

The above discussion shows that the conjecture is true when $\mathfrak{c}_1$ and $\mathfrak{c}_2$ are $\mathfrak{c}$ and the principal class.  To validate the conjecture for another infinite collection of examples, consider, for $k \geq 3$ with $k \not\equiv 2 \pmod{3}$, the primitive form 
\begin{equation*}
	f = (2k-1)x^2 + (2k^2+1) xy + (k^2-k+1) y^2,
\end{equation*}
which has discriminant $(2k^2-2k+3)^2-4$. Denote its equivalence class by $\tilde{\mathfrak{c}}$. It can be shown that $\tilde{\mathfrak{c}}^2 = \mathfrak{c}$. (For instance, set $f=g=k$ and $h=k^2+1$ in \cite[Theorem 2.11]{fL2014}.  The composition will be one reduction step away from the form \eqref{E:specialform}, and the reducing number is $1$.)

It is readily checked that the corresponding sequence is $(k,k-1,2)$, so that $f$ has sum $2k+1$, which does not exceed $a=2k^2-2k+3$.   We can, in fact, describe the entire kneading cycle:
\begin{align*}
	(k,k-1,2) &\mapsto (1,k-2,1,1,k) \mapsto (1,k-3,1,1,k-1,1,1) \mapsto (1,k-4,1,1,k-1,2,1) \cdots\\
	&\mapsto (1,1,1,1,k-1,k-3,1) \mapsto (2,1,k-1,k-2,1) \mapsto (k,k-1,2),
\end{align*}
which we see has caliber $k$.  This confirms Conjecture \ref{C:composition} when $\mathfrak{c}_1 = \mathfrak{c}_2 = \tilde{\mathfrak{c}}$.

It is readily checked that the only primitive forms that are fixed by reduction are those of the form $x^2 + axy + y^2$.  It follows that in Conjecture \ref{C:composition} we have $2 \leq l_1, l_2 \leq n-3$ unless one of the classes $\mathfrak{c}_1$ and $\mathfrak{c}_2$ is the principal class.

\section{Acknowledgement}
The author would like to thank Benjamin Dickman, Franz Lemmermeyer, and his father, Weldon Smith, for helpful comments during the preparation of this manuscript.

\end{document}